\documentclass[12pt]{amsart}

\usepackage{amsmath,amssymb,amsthm}
\usepackage{graphicx}
\usepackage{enumerate}

\usepackage[paperwidth=8.5in,paperheight=11in,
left=1.125in, right=1.125in,top=1.25in, bottom=1.25in]{geometry}

\newtheorem{theorem}{Theorem}[section]
\theoremstyle{plain}

\newtheorem{corollary}[theorem]{Corollary}

\newtheorem{definition}[theorem]{Definition}
\newtheorem{example}[theorem]{Example}

\newtheorem{lemma}[theorem]{Lemma}

\newtheorem{proposition}[theorem]{Proposition}

\numberwithin{equation}{section}

\DeclareMathOperator{\aut}{Aut}
\DeclareMathOperator{\rk}{rank}

\newcommand{\abar}{\bar{\alpha}}
\newcommand{\bbar}{\bar{\beta}}
\newcommand{\cbar}{\bar{\gamma}}

\newcommand{\partialb}{\bar{\partial}_{\flat}}

\newcommand{\Mob}{\text{\rm M\"{o}b}}

\begin{document}
\title[The Schwarzian derivative and M\"{o}bius equation]{The Schwarzian derivative and M\"{o}bius equation on strictly pseudo-convex CR manifolds}
\author{Duong Ngoc Son}
\address{Texas A\&M University at Qatar, Science Program, PO Box 23874, Education City, Doha, Qatar}
\email{son.duong@qatar.tamu.edu}
\date{19 June 2016}
\thanks{The author was partially supported by the Qatar National Research Fund, NPRP project 7-511-1-098.}
\begin{abstract}
The notion of Schwarzian derivative for locally univalent holomorphic functions on complex plane was
generalized for conformal diffeomorphisms by Osgood and Stowe in 1992 \cite{OS}. We shall identify a tensor that may serve as an analogue of the Schwarzian of Osgood and Stowe for CR mappings, and then use the tensor to define and study the CR M\"{o}bius transformations and metrics of pseudo-hermitian manifolds. We shall establish basic properties of the CR Schwarzian and a local characterization of CR spherical manifolds in terms of fully integrability of the CR M\"{o}bius equation. In another direction, we shall prove two rigidity results for M\"{o}bius change of metrics  on compact CR manifolds.
\end{abstract}
\maketitle
\section{Introduction}
In \cite{OS},
Osgood and Stowe introduce the $(0,2)$ symmetric traceless  tensor associated with a $C^2$ function on a Riemannian manifold as follows. 
Let $(M^m,g)$ be a Riemannian manifold of dimension~$m$. If $\varphi \colon M^m \to \mathbb{R}$ is a $C^{2}$ function, then the \emph{Schwarzian tensor} of $\varphi$ is defined to be
\begin{equation}
 B_g(\varphi) = \mathrm{Hess}\, (\varphi) - d\varphi \otimes d\varphi -\tfrac{1}{m}\{\Delta \varphi - \|\mathrm{grad}\, \varphi\|^2\}g,
\end{equation}
where $ \mathrm{Hess}\, (\varphi)$ is the Hessian computed with respect to Levi-Civita connection and $\Delta$ is the Laplacian.
For a conformal
local diffeomorphism $f\colon (M,g) \to (M',g')$ with $\hat{g} = e^{2\varphi}g = f^{\ast}g'$, the \emph{Schwarzian derivative}
of $f$ is defined to be
\begin{equation}
 \mathcal{S}_g(f) = B_g(\varphi).
\end{equation}
This tensor is a generalization to higher dimension of the Schwarzian derivative of a conformal mapping in the
complex plane. Its basic properties include:
\begin{enumerate}[(i)]
 \item Suppose that $h\colon (M,g)\to (M', g')$ and $f\colon (M', g')\to (M'',g'')$, the ``chain rule'' holds:
 \begin{equation}
 \mathcal{S}_g(f\circ h) = \mathcal{S}_{g}(h) + h^{\ast} \mathcal{S}_{g'}(f);
 \end{equation}
 \item $\mathcal{S}_g(f) = 0$ if and only if $f$ preserves the traceless component of Ricci tensor. This means if $R_{ij}$ and $\hat{R}_{ij}$ is the Ricci curvatures of $g$ and $\hat{g} := h^{\ast}g'$, respectively, then 
 \begin{equation}\label{e:tracelessricci}
  \hat{R}_{ij} - R_{ij}\ \text{is a functional multiple of the metric}\ g_{ij}.
 \end{equation} 
\end{enumerate}
As explained in \cite{OS}, when $M=M' = \mathbb{C}$ is the complex plane with Euclidean metric and $f$ is a locally univalent holomorphic mapping, then the Schwarzian derivative tensor $\mathcal{S}_{\mathrm{euc}}(f)$ can
be represented, in the standard coordinates, by a $2\times 2$ matrix
\begin{equation} 
 \mathcal{S}_{\mathrm{euc}}(f)
 =
B(\log |f'|)
=
 \begin{bmatrix}
  \Re S(f) & -\Im S(f)\\
  -\Im S(f) & -\Re S(f)
 \end{bmatrix},
\end{equation} 
where $S(f)$ is the ``classical'' Schwarzian derivative:
\begin{equation} 
 S(f) = \frac{f'''}{f'} - \frac{3}{2}\left(\frac{f''}{f'} \right)^2.
\end{equation} 
Therefore, $\mathcal{S}_{\mathrm{euc}}(f)=0$ if and only if $S(f) =0$. Hence,  the vanishing of $\mathcal{S}_{\mathrm{euc}}(f)$ also characterizes the M\"{o}bius
transformations of complex plane, i.e., those of the form
\begin{equation} 
 f(z) = \frac{az +b}{cz+d}, \quad ad - bc \ne 0.
\end{equation} 
where $a,b,c$ and $d$ are constants. Diffeomorphisms between Riemannian
manifolds satisfying $\mathcal{S}_g(f) =0 $ are also termed M\"{o}bius transformations in \cite{OS}; these transformations of a manifold into itself form a subgroup of the conformal group. Likewise, a  conformal metric $\hat{g} = e^{2\varphi}g$ satisfying $B_{g}(\varphi)=0$ is called a \emph{M\"{o}bius metric} (with respect to $g$). We remark that a conformal change to a M\"{o}bius metric preserves the traceless Ricci tensor as in~\eqref{e:tracelessricci}. It has been studied by many authors (see, e.g., \cite{BR,KR1,KR2} and references therein).

The first purpose of this paper is to introduce and study a notion of Schwarzian derivative for CR mappings between pseudo-hermitian manifolds.
Let $(M, T^{1,0}M,\theta)$ be a strictly pseudo-convex pseudo-hermitian manifold of CR dimension $n$ in the sense of Webster \cite{W}. For a $C^2$ smooth function 
$\varphi \colon M \to \mathbb{R}$, we define the \emph{CR Schwarzian tensor} 
$B_{\theta}(\varphi)$ to be
\begin{multline}\label{e:defsch}
B_{\theta}(\varphi)(Z,W)
 =
\nabla^2 \varphi(Z,W) + \nabla^2 \varphi(W,Z) \\
- 4(\partial_{\flat}\varphi \otimes \partial_{\flat} \varphi + \partialb \varphi \otimes \partialb \varphi)(Z,W) - \tfrac{1}{n}\left(\Delta_{\flat} \varphi\right) G_{\theta}(Z,W).
\end{multline}
where $\Delta_{\flat} \varphi$ is the sub-Laplacian, $G_{\theta}$ is the complexified Levi metric, $Z,W \in T^{1,0}M \oplus T^{0,1}M$. That is, for real vector $X,Y$ in $\mathcal{H}: = \Re (T^{1,0}M $), 
\begin{equation}
G_{\theta}(X,Y) = (d\theta)(X,JY).
\end{equation}
The linear connection $\nabla$ is the Tanaka-Webster connection, which always has torsion, and thus $\nabla^2 \varphi(Z,W)$ is not necessary equal $ \nabla^2\varphi (W,Z)$.

Suppose that $f \colon (M,\theta) \to (M', \theta')$ is a CR (local) diffeomorphism such that 
\begin{equation}
 f^{\ast} \theta' = e^{2\varphi} \theta.  
\end{equation}
In analogy with Osgood and Stowe \cite{OS}, we define the \emph{CR Schwarzian derivative} of $f$ to be
\begin{equation} 
 \mathcal{S}(f) = \mathcal{S}_{\theta}(f) = B_{\theta}(\varphi).
\end{equation} 
We shall say that a CR
diffeomorphism $f \colon (M,\theta) \to (M',\theta')$  is a \emph{M\"{o}bius
transformation} if either (i) $n\geq 2$ and  $\mathcal{S}_{\theta}(f) = 0$; or (ii) 
$n=1$, $\mathcal{S}_{\theta}(f) = 0$ and $\varphi$ is CR-pluriharmonic (the reason why in the case $n=1$ we require $\varphi$ to be CR-pluriharmonic will be clear). Here a function $\psi$ is \emph{CR-pluriharmonic} if locally $\psi$ is a real part of a CR function.  We say that a pseudo-hermitian
structure $\hat{\theta} = e^{2\varphi}$ is a \emph{M\"{o}bius pseudo-hermitian structure with respect to} $\theta$ (or M\"{o}bius metric) if either (i) $n\geq 2$ and $B_{\theta}(\varphi)=0$, or (ii) $n=1$, $B_{\theta}(\varphi) = 0$, and $\varphi$ is CR-pluriharmonic. In any dimension, $\hat{\theta}$ is M\"{o}bius with respect to $\theta$ if and only if the identity mapping $\mathrm{id} \colon (M,\theta)\to (M, \hat{\theta})$ is a M\"{o}bius  transformation. 

In Section~3, we shall establish basic properties of CR Schwarzian derivative which are much similar to those of the Riemannian counterpart. In particular, we shall establish the ``chain rule'', and the relation between the vanishing of the CR Schwarzian and the pseudo-conformal changes of metrics preserving torsion and traceless Webster Ricci tensor.

The second purpose of this paper is to study the \emph{M\"{o}bius equation}, i.e., the equation $B_{\theta}(\varphi)=0$, whose solutions give rises to M\"{o}bius metrics with respect to~$\theta$. For this purpose, we first study the equation on the Heisenberg manifold with its standard pseudo-hermitian structure $\Theta$ (see Section~4). In this situation, the equation $B_{\Theta}(\varphi) =0$ was essentially solved by Jerison and Lee in their study of Yamabe problem on CR manifolds \cite{JLCRYamabe, JLextremal}. In fact, since $\Theta$ is pseudo-Einstein and the pseudo-hermitian torsion vanishes, a M\"{o}bius metric~$\theta$ on the Heisenberg manifold must have constant Webster scalar curvature (see Theorem~\ref{thm:mobric}); all such $\theta$ were found in \cite{JLCRYamabe} and \cite{JLextremal} by solving the M\"{o}bius equation. It turns out that, on $(\mathbb{H}^m,\Theta)$, the equation is \emph{fully integrable}, in the sense that it admits solutions with arbitrary prescribed horizontal gradient at a given point (see Definition~\ref{def:fi}).

Our next result thus obtained is a complete characterization of the pseudo-hermitian manifolds for which the M\"{o}bius equation is fully integrable (Theorem~\ref{thm:cstcur}). As a corollary, we obtain the following characterization of \emph{locally} CR spherical manifolds.
\begin{theorem}\label{thm:intro1}
Let $M$ be a strictly pseudo-convex CR manifold of dimension at least 5 and $p\in M$. Then $M$ is CR spherical near~$p$ if and only if the CR M\"{o}bius equation $B_{\theta}(\varphi) = 0$ is fully integrable at $p$ for some pseudo-hermitian structure $\theta$ defined near~$p$.
\end{theorem}
In Example~\ref{ex:2} below we shall construct a family of \emph{nonspherical} manifolds for which the M\"{o}bius equation admits nonconstant local solutions. The example thus shows that the fully integrability assumption in the theorem above, although rather strong, is essential; it also shows that the corresponding statement in three-dimension does \emph{not} hold.

In Riemannian geometry, a similar result was also obtained by Osgood and Stowe in \cite[Theorem~2.3]{OS}. The result says that the (Riemannian) M\"{o}bius equation is fully integrable at a point if and only if the manifold is of constant curvature near that point. Despite the similarity in the statements, the proof of Theorem~\ref{thm:intro1} shares little similarity with its Riemannian counterpart.

We now turn to the rigidity problem for the M\"{o}bius equation on \emph{compact} pseudo-hermitian manifolds. We shall give two conditions for $(M,\theta)$ so that the solutions to the equation are the constants. Namely, in Section~5, we shall prove the following theorem. (Observe that the required condition only involves the CR structure of~$M$).
\begin{theorem}\label{thm:compact1}
Let $(M,\theta)$ be a compact pseudo-hermitian manifold. If for any $p\in M$, the stability group $\mathrm{Aut}(M,p)$ has zero dimension, then the CR-pluriharmonic (global) solutions to the M\"{o}bius equation $B_{\theta}(\varphi) =0$ are the constants.
\end{theorem}
Here the stability group $\mathrm{Aut}(M,p)$ is the group of germs of CR diffeomorphisms of $M$ fixing~$p$. This theorem seems to have no Riemannian counterpart.

Our final result stated in this introduction is the following theorem whose proof will be given in Section~6.
\begin{theorem} \label{thm:compact2}
  Let $(M,\theta)$ be a compact pseudo-hermitian manifold with nonpositive Webster's Ricci curvature.
  Then the CR-pluriharmonic solutions to $B_{\theta}(\varphi) = 0$ are the constants. The group of M\"{o}bius transformations coincides with the group of homotheties.
\end{theorem}
This theorem is  a CR analogue of Xu's theorem \cite{Xu}, and should be compared with Jerison-Lee theorem on the uniqueness of constant curvature metric on CR manifolds with nonpositive CR Yamabe invariant~\cite{JLCRYamabe}. In fact, a special case
of Theorem~\ref{thm:compact2} (when $M$ is pseudo-Einstein with vanishing torsion) follows from Jerison-Lee's theorem.

As explained in Section~3 below, in the case $n\geq 2$, $B_{\theta}(\varphi)=0$ if and only if the change of metric $\hat{\theta} = e^{2\varphi}\theta$ preserves the pseudo-hermitian torsion and the traceless Ricci tensor. Thus,  Theorem~1.2 and 1.3 say that, under the conditions given in either  one of the theorems, if $\theta$ and $\hat{\theta}$ are two pseudo-hermitian structures having the same pseudo-hermitian torsion and the same traceless Ricci tensor, then $\hat{\theta}/\theta$  must be a constant. We remark that (pseudo)-Riemannian conformal transformations preserving traceless component of Ricci tensor have been studied by many authors (see, e.g., \cite{BR,KR1, KR2} and reference therein).

Let us conclude this introduction by discussing the relation between the CR M\"{o}bius equation considered in this paper and several PDEs  already appeared in the literature. We begin by converting the \emph{nonlinear} system $B_{\theta}(\varphi)=0$ to a linear system as follows. By a well-known characterization of CR-pluriharmonic functions (cf. \cite{L}), if $n\geq 2$, then $B_{\theta}(\varphi) =0$
implies that $\varphi$ is a CR-pluriharmonic
function; locally, we can find a CR-pluriharmonic function $\psi$ such that $\varphi+i\psi$
is CR. If we set $ u = e^{-\varphi}\cos (\psi)$, then we can check by direct calculation that $u$ satisfies
\begin{equation} \label{e:main}
\nabla^2 u(Z,W) + \nabla^2u(W,Z) = \lambda_u G_{\theta}(Z,W), \quad Z, W \in T^{1,0}M \oplus T^{0,1}M,
\end{equation}  
where $G_{\theta}$ is the Levi metric, $\nabla$ is the Tanaka-Webster connection, and $\lambda_u = \frac{1}{n}\Delta_{\flat} u$.
Note that since $\psi$ is \emph{not} unique, neither is~$u$. Conversely, if a solution $u$ to \eqref{e:main} exists, then locally we can find a CR function $G$ such that $u=\Re(G)$. Locally, we can choose $G$ to be non-vanishing. Then $\varphi:=-\Re(\log G) = -\log |G|$ is a well-defined CR-pluriharmonic and satisfies $B_{\theta}(\varphi) = 0$ there. Therefore, each solution to~\eqref{e:main} gives rise to a \emph{local} solution of M\"{obius} equation and a local M\"{o}bius change of pseudo-hermitian structure.

Now we remark that PDEs similar to \eqref{e:main} have already appeared in several papers regarding the Obata-type problem of characterizing the CR sphere related to a Lichnerowicz-type estimate of the first positive eigenvalue of Laplacian (see, e.g., \cite{IV,IV2,LSW,LW} and references therein). In \cite{IV}, Ivanov and Vassilev consider \eqref{e:main} in connection with this problem  and prove that if it has a nontrivial solution $u$ with $\lambda u = Ku$, $K<0$, then $M$ must be CR equivalent to the CR sphere in $\mathbb{C}^{n+1}$ (of appropriate radius) with the ``standard'' pseudo-hermitian structure, \emph{provided} that $M$ is complete and has divergent-free torsion. In \cite{LW}, Li and Wang also prove the CR equivalence to sphere of compact manifolds admitting nontrivial solution $u$ of \eqref{e:main}, $\lambda_u = Ku$, $K<0$ \emph{without} any assumption on the torsion. A partial result on this problem is also obtained by Chang and Chiu \cite{CC}; we refer the readers to \cite{LW,IV,CC} for more details. We also mention \cite{LSW} in which Li, the author, and Wang prove a similar characterization of CR sphere, in connection with eigenvalue estimates for Kohn-Laplacian, in terms of the existence of a nontrivial complex-valued solution to the system  $\nabla^2 f(Z,W) = K fG_{\theta}(Z,W)$, where $K<0$, $Z\in T^{0,1}M$, and $W \in T^{1,0}M \oplus T^{0,1}M$ (Observe the similarity between this system and \eqref{e:main} above.) Therefore, this study, in which we consider \eqref{e:main} with general function $u$ rather than an eigenfunction, is motivated by, and can be regarded as a natural continuation of \cite{IV,LSW,LW}. In particular, Theorem~\ref{thm:intro1} can be regarded as a local version of global characterizations of the CR sphere obtained there.

\subsection*{Acknowledgment}
The author would like to thank S.-Y. Li for introducing him to pseudo-hermitian geometry when he was at University of California, Irvine. He also thanks P. Ebenfelt for teaching and discussions regarding CR manifolds, and N. Mir for interest and encouragement. The author also thanks anonymous referees for useful comments leading to improvement of the precision and readability of this paper.  Part of this work was written while the author enjoyed the hospitality of the Erwin-Schr\"{o}dinger Institute
for Mathematical Physics, University of Vienna in November 2015, which he would like to thank for its support.

\section{Preliminaries}
Throughout this paper, we use notation and terminology similar to those in \cite{L}. For basic notions of CR geometry that are not explained here, see \cite{BER99,DT,L}. Suppose that $M$ is a CR manifold of hypersurface type of CR dimension $n$ (real dimension $2n+1$), i.e., $M$ is endowed with a \emph{formally integrable} complex sub-bundle $T^{1,0}M \subset \mathbb{C}TM$ of complex dimension $n$ such that $T^{1,0} M \cap \overline{T^{1,0}M} = \{0\}$. The real bundle $H:=\Re(T^{1,0}M \oplus T^{0,1}M)$ is then a sub-bundle of the real tangent bundle $TM$ and carries a complex structure $J$, $J^2 =-1$ for which $T^{1,0}M$ is the eigenbundle of the complexification of $J$ corresponding to eigenvalue $i$. Clearly, $H$ has real codimension one in $TM$.

Let $\theta$ be a real 1-form that annihilates $H$. Once chosen, $\theta$ determines a symmetric bilinear form on $H$, the Levi-form, by
\begin{equation}
L_{\theta}(X,Y) = d\theta(X, JY).
\end{equation}
The Levi-form is extended by complex linearity to a bilinear form on $H \otimes \mathbb{C}$ which is Hermitian on $T^{1,0}M \times T^{0,1}M$. We only consider the case when the Levi-form is positive definite and so $\theta$ is \emph{contact}. We say that $M$ is strictly pseudo-convex and $(M,\theta)$ is called a pseudo-hermitian manifold.

On a pseudo-hermitian manifold, Levi-form yields a norm on real or complex tensor bundle over $H$. The Reeb vector field $T$ is uniquely determined by $T \lrcorner\, \theta =1$ and $T \lrcorner\, d\theta =0$.
We use $T$ to extend $J$ to $TM$ by setting $JT =0$; this extension depends on $\theta$ and sometimes denoted by $J_\theta$. Levi-form can also be extended to Webster metric by
\begin{equation}
g_{\theta} (X,Y) = d\theta(X, J_{\theta} Y) + \theta(X) \theta(Y).
\end{equation}

For local computations, we often choose a complex local frame $\{Z_{\alpha}  , Z_{\abar}, T\}$, where $\{Z_{\alpha}  \}$ 
is a local frame for $T^{1,0}M$, and $Z_{\abar} = \overline{Z_{\alpha}  }$. Here we use Greek indices as in \cite{L}; the indices $\alpha, \beta, \gamma$ run from $1$ to $n$ ($n=\dim_{CR}M)$.

Suppose that $\{\theta^{\alpha}, \theta^{\abar}, \theta\}$ is the local admissible coframe dual to $\{Z_{\alpha}, Z_{\abar}, T\}$. We can express the Levi-form as
\begin{equation}
d\theta = i h_{\alpha\bbar} \theta^{\alpha} \wedge \theta^{\bbar},
\end{equation}
Then the component of Levi-form $h_{\alpha\bbar}$ is a Hermitian matrix. We use $h_{\alpha\bbar}$ and its inverse $h^{\bbar \alpha}$ to lower and raise indices in usual way.
The connection 1-form $\omega_{\beta}{}^{\alpha}$ and torsion 1-form $\tau$ is determined by
\begin{align}
d\theta^{\alpha} &= \theta^{\beta} \wedge \omega_{\beta}{}^{\alpha} + A^{\alpha}{}_{\bbar}\, \theta \wedge \theta^{\bbar}\\
0 &= \tau_{\alpha}  \wedge \theta^{\alpha} \\
dh_{\alpha\bbar} &= \omega_{\alpha\bbar} + \omega_{\bbar \alpha}
\end{align}
where $A^{\alpha}{}_{\bbar}$ is the components of Webster torsion; $\tau^{\alpha} = A^{\alpha}{}_{\bbar} \theta^{\bbar}$. The connection form determines covariant derivative on 
$\mathbb{C}\otimes TM$ by
\begin{equation}
\nabla Z_{\alpha} = \omega_{\alpha}{}^{\beta} \otimes Z_{\beta},
\quad
\nabla Z_{\abar} = \omega_{\abar}{}^{\bbar} \otimes Z_{\bbar},
\quad
\nabla T =0.
\end{equation}
The curvature form of the Tanaka-Webster connection can be expressed as
\begin{equation}
d\omega_{\beta}{}^{\alpha} - \omega_{\beta}{}^{\gamma} \wedge \omega_{\gamma}{}^{\alpha}
=
R_{\beta}{}^{\alpha}{}_{\rho\bar{\sigma}} \theta^{\rho}\wedge\theta^{\bar{\sigma}} 
+
W_{\beta}{}^{\alpha}{}_{\rho} \theta^{\rho}\wedge \theta
-
W^{\alpha}{}_{\beta\bar{\rho}} \theta^{\bar{\rho}} \wedge \theta
+
i\theta_{\beta}\wedge \tau^{\alpha} 
-
i\tau_{\beta} \wedge \theta^{\alpha}.
\end{equation}
where the coefficients satisfy
\begin{equation}
R_{\beta\abar\rho\bar{\sigma}} = \overline{R_{\alpha\bbar\sigma\bar{\rho}}} = R_{\abar\beta\bar{\sigma}\rho} = R_{\rho\abar\beta\bar{\sigma}},
\quad
W_{\beta\abar\gamma} = W_{\gamma\abar\beta}.
\end{equation}
The tensor $R_{\beta}{}^{\alpha}{}_{\rho\bar{\sigma}}$ is called Webster curvature tensor. Webster \cite{W} shows that the Chern-Moser tensor \cite{CM}  can be computed as
\begin{align}
S_{\beta}{}^{\alpha}{}_{\gamma\bar{\sigma}}
=
R_{\beta}{}^{\alpha}{}_{\gamma\bar{\sigma}}
- \frac{R_{\beta}{}^{\alpha} h_{\gamma\bar{\sigma}} +R_{\gamma}{}^{\alpha} h_{\beta\bar{\sigma}} + \delta_{\beta}^{\alpha}R_{\gamma\bar{\sigma}} + \delta_{\gamma}^{\alpha} R_{\beta\bar{\sigma}}}{n+2}+ \frac{R(\delta_{\beta}^{\alpha}h_{\gamma\bar{\sigma}}+\delta_{\gamma}^{\alpha}h_{\beta\bar{\sigma}})}{(n+1)(n+2)} .
\end{align}
Here $R_{\alpha}{}^{\delta}{}_{\beta\cbar}$ is the pseudo-Hermitian curvature tensor. The pseudo-Hermitian Ricci tensor is defined by $R_{\alpha\bbar} = R_{\gamma}{}^{\gamma}{}_{\alpha\bbar}$.

Given a smooth function $f$. A choice of $\theta$ allows us to define $\bar{\partial}_{\flat}f $ as a genuine 1-form on $M$. Namely,
\begin{equation}
\bar{\partial}_{\flat} f = f_{\abar} \theta^{\abar}, \quad \partial_{\flat} f = f_{\alpha}\theta^{\alpha}
\end{equation}
so that
\begin{equation}
d f = \bar{\partial}_{\flat} f + \partial_{\flat} f + (Tf) \theta.
\end{equation}

In local calculation, we repeatedly use the following commutation relations
established in \cite{L}, reproduced here for reader's convenience.
\begin{align}
 f_{;\alpha\beta} - f_{;\beta\alpha} &= 0\\
 f_{;\alpha\bbar} - f_{;\bbar \alpha} &= if_{;0} h_{\alpha\bbar} \\
 f_{;0\alpha} - f_{;\alpha 0} &= A_{\alpha\beta} f^{;\beta} \\
 f_{;\alpha\beta \cbar} - f_{;\alpha\cbar\beta} &= if_{;\alpha 0} h_{\beta\cbar} + R_{\alpha}{}^{\delta}{}_{\beta\cbar} f_{;\delta}\\
 f_{;\alpha 0 \bbar} - f_{;\alpha\bbar 0}
 & = f_{;\alpha\gamma} A^{\gamma}{}_{\bbar} + f_{;\gamma} A^{\gamma}{}_{\bbar;\alpha}\\
f_{;\alpha\beta\gamma} - f_{;\alpha\gamma\beta} &= iA_{\alpha\gamma} f_{;\beta} - iA_{\alpha\beta} f_{;\gamma} \label{e:218}
\end{align}
We use indices preceded by a semicolon to indicate covariant derivatives, however, we often omit the semicolon as it poses no danger.

Let us write
\begin{equation}
B_{\theta}(\varphi)
 = B_{\alpha\beta}   \, \theta^{\alpha} \otimes \theta^{\beta} 
+
 B_{\abar\bbar} \, \theta^{\abar} \otimes \theta^{\bbar}
+
 B_{\alpha\bbar}   \, \theta^{\alpha} \otimes \theta^{\bbar} 
+
B_{\abar\beta} \, \theta^{\abar} \otimes \theta^{\beta}
\mod \theta.
\end{equation}
It is immediate from \eqref{e:defsch} that
\begin{align}
 & B_{\alpha\beta}    = 2\varphi_{\alpha\beta}  - 4 \varphi_{\alpha}  \varphi_{\beta}, \quad B_{\abar \bbar} = \overline{B_{\alpha\beta}   } \label{e:sch1} \\
 & B_{\alpha\bbar} = \varphi_{\alpha\bbar} +\varphi_{\bbar \alpha} -\tfrac{1}{n}\left(\Delta_{\flat} \varphi\right)  h_{\alpha\bbar}, \quad B_{\abar \beta} = \overline{B_{\alpha\bbar}   } \label{e:sch2}.
\end{align}
The tensor $B_{\theta}(\varphi)$ is trace-free, i.e.,
\begin{equation} 
\Lambda(B_{\theta}(\varphi))
:= h^{\bbar\alpha}   B_{\theta}(\varphi)(Z_{\alpha}  , Z_{\bbar}) = 0.
\end{equation}

We say that a real-valued function $u$ is \emph{CR-pluriharmonic} if locally $u=\Re F$ for some CR function $F$. Observe that if $B_{\theta}(\varphi) =0$ then from \eqref{e:sch2}, we see that
\begin{equation}
\varphi_{\alpha\bbar} - \tfrac{1}{n} \varphi_{\gamma}{}^{\gamma} h_{\alpha\bbar} = 0.
\end{equation}
By \cite{L}, when $n\geq 2$, $\varphi$ is CR-pluriharmonic function whenever $B_{\theta}(\varphi)=0$. Following Graham and Lee \cite{L,GL}, we define the operator
\begin{equation}\label{e:palpha}
 P_{\alpha}  f = f_{\cbar }{}^{\cbar }{}_{\alpha}  + inA_{\alpha\gamma}f^{\gamma} ,
\end{equation}
which satisfies
\begin{equation}
 \frac{n-1}{n} P_{\alpha}  f = \left(f_{\alpha\bbar} - \tfrac{1}{n} f_{\gamma}{}^{\gamma} h_{\alpha\bbar} \right)^{;\bbar}.
\end{equation}
Therefore, in case $n\geq 2$, $P_{\alpha}  \varphi = 0$ whenever $\varphi$ is CR-pluriharmonic. In dimension three, it is also proved in \cite{L} that $\varphi$ is CR-pluriharmonic if and only if $P_{\alpha}\varphi = 0$.

\section{Elementary properties of Schwarzian tensors and M\"{o}bius transformations on CR manifolds}
In this section, we discuss several elementary properties of Schwarzian for CR mappings.
The discussion is much similar to \cite[Section~2]{OS}, in which the author proved several properties for Schwarzian tensor on Riemannian manifolds.
\begin{proposition}\label{lem:additive}
Let  $\varphi, \sigma \colon M \to \mathbb{R}$ be smooth functions on $(M,\theta)$.
 Then
 \begin{equation} \label{e:additive}
  B_{\theta}(\varphi+\sigma) = B_{\theta}(\varphi) + B_{\hat{\theta}}(\sigma),
 \end{equation} 
where $\hat{\theta} = e^{2\varphi} \theta$.
\end{proposition}
\begin{proof}
Let $\{Z_{\alpha} \}$ be a local holomorphic frame and $\{\theta^{\alpha}\}$
its dual admissible coframe for $\theta$. For $\hat{\theta} = e^{2\varphi}\theta$, we put
\begin{equation}
\hat{\theta}^{\delta}   = e^{\varphi}(\theta^{\delta}   + 2i \varphi^{\delta}   \theta).
\end{equation}
Then $\{\hat{\theta}^{\alpha} \}$ is an admissible coframe for $\hat{\theta}$, with the same Levi matrix, i.e. $\hat{h}_{\alpha\bbar}  = h_{\alpha\bbar} $,  and dual to the holomorphic frame $\{\hat{Z}_{\delta} = e^{-\varphi} Z_{\delta}\}$.
Lee proved in \cite{L86} that the connection form $\hat{\omega}_{\beta}{}^{\alpha} $ for $\hat{\theta}$ is given by
\begin{align}
 \hat{\omega}_{\beta}{}^{\alpha}  
  = \omega_{\beta}{}^{\alpha} & + 2(\varphi_{\beta}\theta^{\alpha} - \varphi^{\alpha}  \theta_{\beta})
    + \delta_{\beta}^{\alpha} (\varphi_{\gamma}  \theta^{\gamma}  - \varphi^{\gamma}  \theta_{\gamma} ) \notag \\
    & + i(\varphi^{\alpha} {}_{\beta} +\varphi_{\beta}{}^{\alpha}  + 4\varphi_{\beta}\varphi^{\alpha}  +4\delta_{\beta}^{\alpha} \varphi_{\gamma}  \varphi^{\gamma} )\,\theta.
\end{align}
We get the following transformation rules for Christoffel symbols (see also \cite[p.~137]{DT})
\begin{align}
 \hat{\Gamma}_{\bbar\alpha} ^{\gamma}  = e^{-\varphi}(\Gamma_{\bbar\alpha} ^{\gamma}  - 2\varphi^{\gamma} h_{\alpha\bbar}  - \delta_{\alpha} ^{\gamma} \varphi_{\bbar})\\
 \hat{\Gamma}_{\beta\alpha} ^{\gamma}  = e^{-\varphi}({\Gamma}_{\beta\alpha} ^{\gamma}  +2\delta_{\beta}^{\gamma} \varphi_{\alpha}  + \delta_{\alpha} ^{\gamma}  \varphi_{\beta}).
\end{align}
Here, $\hat{\Gamma}_{\beta\alpha}^{\gamma}$ are the Christoffel symbols of the Tanaka-Webster connection $\hat{\nabla}$ of $\hat{\theta}$, evaluated with respect to
the frame $\hat{Z}_{\alpha}= e^{-\varphi}Z_{\alpha} $. Therefore,
\begin{equation}
\hat{\nabla}_{\beta}\hat{\nabla}_{\alpha}  \sigma
= e^{-2\varphi} \left\{{\nabla}_{\beta}{\nabla}_{\alpha} \sigma -2\sigma_{\alpha} \varphi_{\beta} - 2\sigma_{\beta} \varphi_{\alpha}  \right\}.
\end{equation}
We then compute (modulus $\theta$)
\begin{align}
 B_{\hat{\theta}}(\sigma)_{\alpha\beta}\, \hat{\theta}^{\alpha} \otimes \hat{\theta}^{\beta}
  & =
    2(\hat{\nabla}_{\beta}\hat{\nabla}_{\alpha}  \sigma - 2(\hat{\nabla}_{\alpha}  \sigma)(\hat{\nabla}_{\beta} \sigma))\, \hat{\theta}^{\alpha} \otimes \hat{\theta}^{\beta}\notag \\
  & = 
    2(\sigma_{\alpha\beta}  - 2\sigma_{\alpha} \sigma_{\beta} - 2\sigma_{\alpha} \varphi_{\beta} - 2\sigma_{\beta}\varphi_{\alpha} )\, \theta^{\alpha} \otimes \theta^{\beta} \mod \theta,
\end{align}
where $\varphi_{;\alpha\beta}$, etc.,  indicate derivatives with respect to $\nabla$. Consequently, modulus $\theta$, 
\begin{align}\label{e:adda}
B_{\theta}(\varphi)_{\alpha\beta}\, & \theta^{\alpha} \otimes \theta^{\beta} +  B_{\hat{\theta}}(\sigma)_{\alpha\beta}\,  \hat{\theta}^{\alpha} \otimes \hat{\theta}^{\beta} \notag \\
 & = 2(\varphi_{;\alpha\beta}  - 2\varphi_{\alpha} \varphi_{\beta} +\sigma_{;\alpha\beta}  - 2\sigma_{\alpha} \sigma_{\beta} - 2\sigma_{\alpha} \varphi_{\beta} - 2\sigma_{\beta}\varphi_{\alpha} )\, \theta^{\alpha} \otimes \theta^{\beta}  \notag  \\
 & = 2\{(\varphi+\sigma)_{;\alpha\beta}  - 2(\varphi+ \sigma)_{;\alpha} (\varphi+\sigma)_{;\beta}\}\, \theta^{\alpha} \otimes \theta^{\beta} \notag\\
 & = B_{\theta}(\varphi+\sigma)_{\alpha\beta}\, \theta^{\alpha} \otimes \theta^{\beta}.
\end{align}
To calculate $B_{\hat{\theta}}(\sigma)_{\alpha\bbar} $, we note that
\begin{equation}
\hat{\nabla}_{\bbar} \hat{\nabla}_{\alpha}  \sigma = e^{-2\varphi}\{\sigma_{;\alpha\bbar}  + 2\varphi^{\gamma} \sigma_{\gamma}  h_{\alpha\bbar} \}.
\end{equation}
Consequently,
\begin{equation}
\hat{\Delta}_{\flat} \sigma = e^{-2\varphi}\left(\Delta_\flat \sigma +4n\, \varphi^{\gamma}  \sigma_{\gamma} \right).
\end{equation}
Therefore,
\begin{align}
B_{\hat{\theta}}(\sigma)_{\alpha\bbar}\, \hat{\theta}^{\alpha}  \otimes \hat{\theta}^{\bbar}
 & =  \left\{\hat{\nabla}_{\bbar} \hat{\nabla}_{\alpha}  \sigma + \hat{\nabla}_{\alpha} \hat{\nabla}_{\bbar}  \sigma - \tfrac{1}{n}\left(\hat{\Delta}_{\flat}\sigma\right) \hat{h}_{\alpha\bbar} \right\}\, \hat{\theta}^{\alpha}  \otimes \hat{\theta}^{\bbar} \notag \\
 & = e^{-2\varphi}\biggl\{\sigma_{;\alpha\bbar} +\sigma_{;\bbar\alpha} +4\varphi^{\gamma} \sigma_{\gamma}  h_{\alpha\bbar} \notag \\
 & \qquad \qquad \qquad - \tfrac{1}{n}\left({\Delta_{\flat} \sigma}\right) h_{\alpha\bbar}  - 4\varphi^{\gamma}  \sigma_{\gamma}  h_{\alpha\bbar}      \biggr\} \hat{\theta}^{\alpha}  \otimes \hat{\theta}^{\bbar} \notag  \\
  & = B_{\theta}(\sigma)_{\alpha\bbar}  \, \theta^{\alpha} \otimes \theta^{\bbar} \mod \theta.
\end{align}
Hence,
\begin{equation} \label{e:addb}
B_{\theta}(\varphi)_{\alpha\bbar}\, \theta^{\alpha} \otimes \theta^{\bbar} + B_{\hat{\theta}}(\sigma)_{\alpha\beta}\, \hat{\theta}^{\alpha} \otimes \hat{\theta}^{\bbar}
 = B_{\theta}(\varphi + \sigma)_{\alpha\bbar}\theta^{\alpha} \otimes \theta^{\bbar} \mod \theta.
\end{equation}
From \eqref{e:adda} and \eqref{e:addb}, we easily obtain the \eqref{e:additive}, as desired.
\end{proof}
We consider compositions of CR mappings. Suppose $h \colon (M,\theta) \to (M',\theta')$ and 
$f \colon (M',\theta') \to (M'',\theta'')$ are CR mappings such that
\begin{equation}
 h^{\ast} \theta' = e^{2\varphi}\theta, \quad f^{\ast} \theta'' = e^{2\sigma}\theta'.
\end{equation}
Then 
\begin{equation}
(f\circ h)^{\ast}\theta'' = e^{2(\varphi + \sigma \circ h)}\theta.
\end{equation}
Therefore, by Lemma~\ref{lem:additive},
\begin{equation}
 \mathcal{S}_{\theta}(f\circ h) = B_{\theta}(\varphi + \sigma \circ h) = B_{\theta}(\varphi) + B_{e^{2\varphi}\theta}(\sigma\circ h).
\end{equation}
On the other hand, $h\colon (M,e^{2\varphi}\theta) \to (M', \theta')$ is a pseudo-hermitian diffeomorphism and therefore,
\begin{equation}
 B_{e^{2\varphi}\theta}(\sigma\circ h) = h^{\ast} B_{\theta'}(\sigma).
\end{equation}
We then obtain the following ``chain rule'':
\begin{equation}\label{e:chainrule}
  \mathcal{S}_{\theta}(f\circ h) =  \mathcal{S}_{\theta}(h) + h^{\ast}  \mathcal{S}_{\theta'}(f).
\end{equation}
This identity is the CR analogue of  (2.2) in \cite{OS}; the identity immediately implies that composites of CR M\"{o}bius 
transformations are M\"{o}bius. Put $h=f^{-1}$ into \eqref{e:chainrule}, we obtain, since $\mathcal{S}_{\theta'}(\mathrm{id})=0$, that
\begin{equation}
\mathcal{S}(f^{-1}) = -(f^{-1})^{\ast} \mathcal{S}(f).
\end{equation}
This implies that the inverse of a CR M\"{o}bius transformation is M\"{o}bius.

It can be also proved that if $\hat{\theta}$ is a M\"{o}bius metric with respect to $\theta$, then 
\begin{equation}
\mathcal{S}_{\theta}(f) = \mathcal{S}_{\hat{\theta}}(f),
\end{equation}
Detail is left to the reader.

We therefore obtain the following theorem, which is an analogue of Theorem~2.2 in \cite{OS} for CR geometry.
\begin{theorem}\label{thm:group}
 Composites and inverses of M\"{o}bius transformations are M\"{o}bius.
The M\"{o}bius  transformations of a pseudo-hermitian manifolds $(M,\theta)$ into itself form
a group $\Mob(M,\theta)$. 
If $\tilde\theta = e^{2\sigma}\theta$ is a M\"{o}bius  metric with respect to $\theta$, then $\Mob(M,\tilde{\theta}) = \Mob(M,\theta)$.
\end{theorem}

The M\"{o}bius group is then a (closed) Lie subgroup of the CR automorphism group $\mathrm{Aut}(M)$ and contains the group of homotheties. Therefore, 
\begin{equation}
\mathrm{Psh}(M,\theta) \subset \mathrm{Hty}(M,\theta) \subset \Mob(M,\theta) \subset \mathrm{Aut}(M).
\end{equation}
Here the subgroups  $\mathrm{Psh}(M,\theta)$ and $\mathrm{Hty}(M,\theta)$ are the pseudo-hermitian diffeomorphisms (those $f\in \aut(M)$ such that $f^{\ast}\theta = \theta$),
and the homotheties (those $f$ such that $f^{\ast} \theta = \lambda \theta$, $\lambda$ being a positive constant).

When $n\geq 2$, a pseudo-hermitian manifold $(M,\theta)$ is said to be \emph{pseudo-Einstein} if
the traceless Ricci tensor $R^{\circ}_{\alpha\bbar}:= R_{\alpha\bbar}  - \frac{1}{n}R h_{\alpha\bbar} $ vanishes. This vanishing condition is vacuous when $n=1$. A three-dimension pseudo-hermitian manifold is pseudo-Einstein if 
\begin{equation}
W_{\alpha} :=R_{;\alpha}  - i A_{\alpha\beta}{}^{;\beta}  = 0.
\end{equation}
The form $W_{\alpha}$ was introduced by Hirachi in \cite{Hi} in connection with the Szeg\"{o} kernel on three-dimensional manifolds. The definition of pseudo-Einstein structure in three-dimension above was introduced recently in \cite{CY}. In all dimension we say that $\theta$ is \emph{Einstein} if $\theta$ is torsion-free and pseudo-Einstein (cf. \cite{FL,Wang}). Observe that if $\theta$ is Einstein, then the Webster scalar curvature is constant.

A geometric interpretation of M\"{o}bius change of metrics is the following ``well-known'' proposition.
\begin{proposition}\label{thm:mobric}
 Suppose that $(M,\theta)$ is a pseudo-hermitian manifolds and $\hat{\theta} = e^{2\varphi}  \theta$.
 Then $\hat{\theta}$ is a M\"{o}bius  metric with respect to $\theta$ if and only if either
\begin{enumerate}[(i)]
 \item $n\geq 2$, 
\(
  \hat{A}_{\alpha\beta}  = A_{\alpha\beta}, \  \hat{R}^{\circ}_{\alpha\bbar} = R^{\circ}_{\alpha\bbar}%
\), or
 \item $n=1$, 
\(
 \hat{A}_{\alpha\beta}  = A_{\alpha\beta}, \  \hat{W}_{\alpha}  = e^{-3\varphi} W_{\alpha} .
\)
\end{enumerate}
Here $\hat{A}_{\alpha\beta} $ and $\hat{R}^{\circ}_{\alpha\bbar} $ are torsion and
traceless Ricci of $\hat{\theta}$ evaluated with respect to coframe $\{\hat{\theta}^{\alpha}  = \theta^{\alpha} + 2i\varphi^{\alpha}  \theta\}$. In particular, suppose that $\theta$ is Einstein, then $\hat{\theta}$ is M\"{o}bius with respect to $\theta$ if and only if $\hat{\theta}$ is also Einstein.
\end{proposition}
\begin{proof} Standard formulas for transformations of Webster Ricci and torsion are established in \cite{L}. Namely, for $\hat{\theta} = e^{2\varphi}\theta$, 
\begin{align}
 \hat{R}_{\alpha\bbar}  & = R_{\alpha\bbar}  - (n+2)(\varphi_{\alpha\bbar}  + \varphi_{\bbar \alpha})
 -
 (\Delta_{\flat}\varphi + 4(n+1)|\partialb \varphi|^2)h_{\alpha\bbar} \\
 \hat{A}_{\alpha\beta}  &= A_{\alpha\beta} + 2i\varphi_{\alpha\beta}  - 4i\varphi_{\alpha}  \varphi_{\beta},
\end{align}
These equations can be rewrite as
\begin{align}
i\hat{A}_{\alpha\beta}
 & =
iA_{\alpha\beta} 
-
B_{\theta}(\varphi)_{\alpha\beta}   \\
\hat{R}_{\alpha\bbar} 
 & =
 {R}_{\alpha\bbar} 
 - 
(n+2) B_{\theta}(\varphi)_{\alpha\bbar} 
-\tfrac{2n+2}{n}\left(\Delta_{\flat} \varphi + 2n|\bar{\partial}_{\flat} \varphi |^2 \right)h_{\alpha\bbar} 
\end{align}
Note in passing that in term of CR Schouten tensor, the 2nd equation can be written as
\begin{align}
\hat{P}^{\circ}_{\alpha\bbar}
  =  P^{\circ}_{\alpha\bbar} - B_{\theta}(\varphi)_{\alpha\bbar} .
\end{align}
where $\hat{P}_{\alpha\bbar}$ is the CR Schouten tensor $P_{\alpha\bbar}$, i.e.,
\begin{equation}
{P}_{\alpha\bbar}: = \left(\tfrac{1}{n+2}\right) \left(R_{\alpha\bbar} -\tfrac{1}{2n+2}Rh_{\alpha\bbar}\right),
\end{equation}
and $P^{\circ}_{\alpha\bbar}$ is the traceless component of $P_{\alpha\bbar}$. Then (i) 
immediately follows. 

When $n=1$, it was proved in \cite{Hi} that 
\begin{equation}
\hat{W}_{\alpha} 
= e^{-3\varphi}\left(W_{\alpha} -6P_{\alpha}\varphi\right).
\end{equation}
where $P_{\alpha}$ is the Graham-Lee operator characterizing the CR-pluriharmonic functions in three-dimension. 
In particular, $\hat{W}_{\alpha} = e^{-3\varphi}W_{\alpha}$ is equivalent to that $P_{\alpha} \varphi =0$, which in turn, is equivalent to $\varphi$ being CR-pluriharmonic by \cite{L}. Then (ii) follows immediately.
\end{proof}
From Theorem~\ref{thm:mobric}, we see that finding an Einstein metric on a CR manifold amounts  to solving the \emph{nonhomogeneous} M\"{o}bius equation. In fact, consider the equation
\begin{equation}\label{e:mobius}
B_{\theta}(\varphi) = E,
\end{equation}
where $E$ is a traceless tensor of the same type. Suppose that $\theta$ is a pseudo-hermitian
structure on $M$ and $E$ is given by
\begin{equation}\label{e:tsem}
E_{\alpha\beta}  = -iA_{\alpha\beta}, \quad E_{\alpha\bbar}  = \tfrac{1}{n+2} \left(R_{\alpha\bbar}  - \tfrac{1}{n}R h_{\alpha\bbar} \right),
\end{equation}
then a solution to $B_{\theta}(\varphi) = E$ gives rise to an Einstein metric on $M$.
\begin{corollary}
Let $(M,\theta)$ be a pseudo-hermitian manifold. Suppose that $B_{\theta}(\varphi) = E$ is solvable by a function $\varphi$, where $E$ is given by~\eqref{e:tsem}. Then $\hat{\theta} = e^{2\varphi}\theta$ is Einstein;  and $(M,\theta)$ is locally equivalent to a unit circle bundle in a Hermitian line bundle over a K\"{a}hler-Einstein manifold.
\end{corollary}
Similar to Riemannian case in \cite{OS}, Proposition~\ref{lem:additive} shows that if $\tilde{\varphi}$ is a particular solution of the \emph{nonlinear} equation \eqref{e:mobius},
then determining all solutions reduces to solving the homogeneous equation $B_{\hat{\theta}}(\varphi) =0$, where $\hat{\theta} = e^{2\tilde{\varphi}} \theta$.
From now on, we focus on the homogeneous equation $B_{\theta}(\varphi) =0$.
Also, following Osgood and Stowe, we introduce the following definition.
\begin{definition}\label{def:fi}
We say that \eqref{e:mobius} is \emph{fully integrable} at $p\in M$ if for every (1,0)-form $\omega$ near~$p$, there exists a local solution $\varphi$ satisfying $\omega|_p = \partial_{\flat} \varphi|_p$.
\end{definition}
We observe that the fully integrability is an open condition: if $B_{\theta}(\varphi)=0$ is fully integrable at~$p$, then it is so for all points near $p$. To see this, we can  instead consider the (locally) equivalent linear system \eqref{e:main}. If \eqref{e:main} is fully integrable at $p$, then we can find $n$ solutions $u^1, u^2, \dots , u^n$ such that $\{\partial_{\flat} u^j|_p \mid j=1,2,\dots n\}$ forms a basis of $(T^{1,0})^{\ast}_pM$. There is a neighborhood $V$ of $p$ such that $u^j$ is defined on $V$ for all $j$, and for all $q\in V$, $\{\partial_{\flat} u^j|_q \mid j=1,2,\dots n\}$ is a basis $(T^{1,0})^{\ast}_qM$. Therefore, for any (1,0)-form $\omega$ near $q$, the equality $\omega|_q = \partial_{\flat} u|_q$ is satisfied for a suitable linear combination $u$ of $u^j$. That $u$ is a solution of \eqref{e:main} follows from the linearity of~\eqref{e:main}.
\section{M\"{o}bius metrics on Heisenberg manifold}
It is helpful to have explicit formulas for solutions to the M\"{o}bius equation on some model manifolds, e.g., Heisenberg manifold. In Heisenberg model, the solutions were essentially found by Jerison and Lee in their celebrated paper \cite{JLCRYamabe} on the Yamabe problem on CR manifolds.
 
 The Heisenberg group $\mathbb{H}^n$ is the set
 $\mathbb{C}^n\times \mathbb{R}$ with coordinates $(z,t)$ and group law
 \begin{equation}
 (z,t)(\zeta,\tau)
 =
 \left( z+\zeta, t+\tau + 2\Im z\cdot \bar{\zeta}\right).
 \end{equation}
For a holomorphic frame, we take the left-invariant vector fields
\begin{equation}
Z_{\alpha} = \frac{\partial}{\partial z^{\alpha}} + i \bar{z}^{\alpha} \frac{\partial }{\partial t},
\end{equation}
which make $\mathbb{H}^n$ a strictly pseudo-convex CR manifold. The ``standard'' contact form is
\begin{equation}
\Theta = \tfrac{1}{2}dt + \tfrac{1}{2}\sum (iz^{\alpha} d\bar{z}^{\alpha} - i\bar{z}^{\alpha}dz^{\alpha}).
\end{equation}
Observe that the Levi matrix is identity (i.e., $h_{\alpha\bbar}  = \delta_{\alpha\beta}$) and the Reeb vector field is $T=\partial /\partial t$.
\begin{proposition}\label{prop:heisenberg}
 The equation $B_{\Theta}(\varphi) = 0$ is fully integrable on $\mathbb{H}^n$.
 The set of CR pluriharmonic solutions consists of functions of the form
 \begin{equation} \label{e:heisenbergsol}
  \varphi(z, t) = -\log |\kappa t + i \kappa |z|^2 + z \cdot \mu + \lambda | + C.
 \end{equation}
 where $\kappa, \lambda \in \mathbb{C}$, $\mu \in \mathbb{C}^n$, and $C\in \mathbb{R}$. Consequently, the M\"{o}bius group of $\mathbb{H}^n$ is equal to the group of CR automorphisms.
\end{proposition}
\begin{proof} The proof is essentially given in \cite[p. 10]{JLextremal}. We will follows the same calculation. Suppose that $\varphi$ is a CR pluriharmonic solution to $B_{\Theta}(\varphi)=0$ (when $n\geq2$, all solutions are automatically CR pluriharmonic); 
we can find a real-valued function $\psi$ such that $\varphi+i\psi$ is CR. This implies $G:= e^{-\varphi - i\psi}$ is CR. By direct calculation
\begin{equation}
G_{\alpha\beta} = e^{-\varphi - i\psi}(4\varphi_{\alpha} \varphi_{\beta} - 2\varphi_{\alpha\beta}) = -G B_{\Theta}(\varphi)_{\alpha\beta}
=0.
\end{equation}
Since $G$ is CR, one has
\begin{equation}
G_{0\abar} = G_{\abar 0} = 0.
\end{equation}
Therefore,
\begin{equation}
G_{0\alpha} = -\tfrac{i}{n}(G_{\beta}{}^{\beta}{}_{\alpha} - G_{\bbar}{}^{\bbar}{}_{\alpha})
=
-\tfrac{i}{n}(G_{\beta\alpha}{}^{\beta} - iG_{\alpha 0}) = -\tfrac{1}{n} G_{\alpha 0}.
\end{equation}
Therefore, $G_{0\alpha} = G_{0\abar}=0$ and thus $G_0$ is a constant. Let $w = t+i|z|^2$ and consider the CR function $K = G - G_0w$. Then direct calculation shows that $\partial K/\partial t =0$ and so $K$ is a holomorphic in $\{z^{\alpha}\}$. Moreover,
\begin{equation}
\frac{\partial ^2 K}{\partial z^{\alpha} \partial z^{\beta}} = K_{\alpha\beta} = G_{\alpha\beta} - G_0w_{\alpha\beta} = 0.
\end{equation}
Hence, $K$ is a linear polynomial in $z$ and so
\begin{equation}
G = G_0w + z\cdot \mu + \lambda.
\end{equation}
Therefore,
\begin{equation}
\varphi(z,t) = -\log |G| = -\log |\kappa t+i\kappa |z|^2 + z\cdot \mu + \lambda |, \quad \kappa = G_0,
\end{equation}
as desired.

It is well-known (see \cite{JLCRYamabe}) that then for any CR automorphism $\Psi$ of $\mathbb{H}^n$, $\Psi^{\ast} \Theta = e^{2\varphi} \Theta$, where $\varphi$ is in the form \eqref{e:heisenbergsol} with $4\Im (\bar{\kappa}\lambda) >|\mu|^2$. In particular, $\Psi$ is a M\"{o}bius transformation of $(\mathbb{H}^n, \Theta)$.

To see that the M\"{o}bius equation is fully integrable, we compute,
\begin{equation}
 \varphi_{\alpha} = -G^{-1}\left(i \kappa \bar{z}^{\alpha} + \tfrac{1}{2}{\mu}^{\alpha}\right).
\end{equation}
It then follows that the solution $\varphi$ can be chosen so that the $\partial_\flat \varphi$ at a point for which $G\ne 0$ equals any prescribed value.
\end{proof}
The solutions thus obtained can be transferred to the CR sphere via the Cayley transform (see \cite{JLCRYamabe} for detail). We remark that a related result on the sphere was obtained in \cite{LiLuk}. Namely, Li and Luk determined all $\varphi$ for which $e^{2\varphi} \theta_0$ has vanishing pseudo-hermitian torsion. Here $\theta_0$ is the standard pseudo-hermitian structure on the sphere. Those $\varphi$ are for which $e^{-2\varphi}$ is harmonic quadratic polynomials \cite[Theorem~1.2]{LiLuk}).
\begin{corollary}
Suppose $f\colon (M,\theta) \to (\mathbb{S}^{2n+1},\theta_0)$ be a CR (local) diffeomorphism. Then
\begin{equation}
\mathcal{S}_{\theta}(f) = \mathcal{S}_{\theta}(\Psi\circ f)
\end{equation}
for all $\Psi \in \mathrm{Aut}(\mathbb{S}^{2n+1})$.
\end{corollary}
If $\varphi$ is of the form \eqref{e:heisenbergsol} and $\theta = e^{2\varphi}\Theta$, then $\theta$ is Einstein. Moreover, we show in Theorem~\ref{thm:cstcur} below that $\theta$ has Webster curvature tensor of the form
\begin{equation}\label{e:cstcur}
R_{\beta\cbar\alpha\bar{\sigma}} 
= 
\eta(\delta_{\alpha\gamma} \delta_{\beta\sigma}+\delta_{\alpha\sigma}  \delta_{\beta\gamma}),
\end{equation}
where 
\begin{equation}
\eta = -\frac{R}{n(n+1)}.
\end{equation}
Here $R$ is the Webster scalar curvature, which can be computed using
\begin{equation}
R 
= -2e^{-2\varphi}(n+1)\left( \Delta_{\flat}\varphi +2n|\bar{\partial}_{\flat} \varphi|^2 \right).
\end{equation}
By direct calculation, one get
\begin{equation}
\varphi_{;\alpha\bbar} = -i\kappa G^{-1} \delta_{\alpha\beta}.
\end{equation}
Therefore,
\begin{equation}
R = n(n+1)(4\Im(\bar{\kappa} \lambda) - |\mu|^2).
\end{equation}
Observe that if $R>0$, then $4\Im(\bar{\kappa}\lambda) - |\mu|^2>0$ and $\varphi$ is defined on all $\mathbb{H}^n$; this is the case considered in \cite{JLCRYamabe}.

Our main finding in this section is the following theorem.
\begin{theorem}\label{thm:cstcur}
	Let $(M,\theta)$ have dimension at least 5 and $p\in M$. Then the CR M\"{o}bius equation $B_{\theta}(\varphi) = 0$ is fully integrable at $p$ if and only if $(M,\theta)$ is locally equivalent to a pseudo-hermitian space form of vanishing pseudo-hermitian torsion near~$p$.
\end{theorem}
The pseudo-hermitian space forms of vanishing torsion were described in \cite{W} as hypersurfaces in $\mathbb{C}^{n+1}$ (see \cite[Section~1.5]{DT}). There are three models:
\begin{align}
Q_0\colon r_0(z,w) & = h_{\alpha\bbar}z^{\alpha}z^{\bbar} + \tfrac{i}{2}(w-\bar{w}) = 0, \\
Q_+(c)\colon r_+(z,w) & = h_{\alpha\bbar}z^{\alpha}z^{\bbar} + |w|^2 - c = 0,\\
Q_-(c)\colon r_-(z,w)  & = h_{\alpha\bbar}z^{\alpha}z^{\bbar}  - |w|^2 + c = 0.
\end{align}
Here the pseudo-hermitian structures $\theta$ is given by $\theta = \iota^{\ast}\frac{i}{2}(\bar{\partial} r - \partial  r)$. Note that
each $M\in \{Q_0, Q_+(c), Q_-(c)\}$ is locally CR spherical. Thus, Theorem~\ref{thm:intro1} follows from Theorem~\ref{thm:cstcur}. In strictly pseudo-convex case, $h_{\alpha\bbar}$ is positive definite; $Q_0$ is equivalent to the Heisenberg manifold described above. Also, if $M$ is torsion-free, then $M$ is a pseudo-hermitian space form if and only if $M$ has constant sectional curvature (see \cite{Barletta,W} for details).

For the proof of Theorem~\ref{thm:cstcur}, we need the following lemma.
\begin{lemma}\label{lem:localtor}
Suppose that $(M,\theta)$ is a pseudo-hermitian manifold of dimension at least~5.
\begin{enumerate}[(i)]
 \item Suppose that $\varphi$ is a solution to $B_{\theta}(\varphi)=0$. If $p$ is a regular point of $\varphi$, then $\rk (A_{\alpha\beta}) \leq 1$ near $p$.
 \item Suppose that  $\varphi_1, \varphi_2$ are two solutions to $B_{\theta}(\varphi)=0$. If $p$ is a regular point of $(\varphi_1,\varphi_2)$, i.e, $\partialb \varphi_1 \wedge \partialb \varphi_2 \ne 0$, then  $A_{\alpha\beta} = 0$ near $p$.
\end{enumerate}
\end{lemma}
\begin{proof} We use similar idea as in \cite{LW, LSW}. We remark that the paper \cite{IV} also contains various formulas (e.g., in the proof of Lemma~4.1 in that paper) leading to a proof of this lemma.

Suppose that $B_{\theta}(\varphi) = 0$ then,
\begin{equation} 
\varphi_{;\alpha\beta}  = 2\varphi_{\alpha}  \varphi_{\beta}.
\end{equation} 
Taking derivative, we obtain
\begin{equation}
 \varphi_{;\alpha\beta\gamma} = 2\varphi_{;\alpha\gamma}\varphi_{\beta} + 2\varphi_{\alpha} \varphi_{;\beta\gamma} = 8\varphi_{\alpha} \varphi_{\beta}\varphi_{\gamma}.
\end{equation}
Therefore, $\varphi_{;\alpha\beta\gamma} = \varphi_{;\alpha\gamma\beta}$. Using commutation relation~\eqref{e:218}, we deduce that
\begin{equation}
 A_{\alpha\beta}\varphi_{\gamma}  = A_{\alpha\gamma}\varphi_{\beta}
\end{equation}
This implies that whenever $\partialb \varphi \ne 0$,
\begin{equation}
 A_{\alpha\beta} = \zeta\, \varphi_{\alpha}  \varphi_{\beta},
\end{equation}
where the function $\zeta$ is given by
\begin{equation}
\zeta = |\bar{\partial}_{\flat} \varphi|^{-4} A_{\gamma\sigma}\varphi_{\bar{\gamma}} \varphi_{\bar{\sigma}}.
\end{equation}
Therefore, the matrix $(A_{\alpha\beta})$ is a scalar multiple of the rank-1 matrix $(\varphi_{\alpha} ) \otimes (\varphi_{\beta})$. This completes the proof of (i). The proof of (ii) also follows, as under the assumption of (ii), $(A_{\alpha\beta})$ must be a scalar multiple of two linearly independent (in the space of matrices) rank-1 matrices.
\end{proof}
 \begin{proof}[Proof of Theorem~\ref{thm:cstcur}]
We first prove the necessity condition. Suppose that $B_{\theta}(\varphi)=0$ is fully integrable at $p$ and hence on a neighborhood of $p$.
Applying Lemma~\ref{lem:localtor}~(ii), we see immediately that  the torsion $A_{\alpha\beta}$ must vanish near $p$.

Suppose that $\varphi$ is a nonconstant solution to $B_{\theta}(\varphi) =0$. Let $u=e^{-\varphi} \cos (\psi)$, where $\psi$ is a function such that $\varphi+i\psi$ is CR.  Then $u$ is CR-pluriharmonic and satisfies $u_{\alpha\beta}=0$. There is a function $v$ defined near $p$ such that
\begin{equation}
u_{\alpha\bbar}  = v\delta_{\alpha\beta} .
\end{equation}
Moreover, since $A_{\alpha}{}^{\bbar} =0$,
\begin{equation}
0 = P_{\alpha} u = u_{\bbar}{}^{\bbar}{}_{;\alpha} = n \bar{v}_{\alpha}.
\end{equation}
Thus $v$ is CR. 
Therefore,
\begin{equation} 
iu_{0;\alpha} \delta_{\beta\gamma} = (u_{\gamma\bbar} - u_{\bbar\gamma})_{;\alpha} 
=
\delta_{\beta\gamma} (v - \bar v) _{;\alpha} 
=
\delta_{\beta\gamma} v_{\alpha} 
\end{equation} 
Consequently, since $A_{\alpha}{}^{\bbar} =0$,
\begin{equation} 
v_{\alpha}  = iu_{0;\alpha} = i u_{\alpha;0} +i A_{\alpha} {}^{\bbar} u_{\bbar}
= i u_{\alpha;0}.
\end{equation} 
On the other hand, from \(u_{\alpha\beta}  =0\), we have
\begin{align}\label{e:wecur}
 0 = u_{\alpha\beta;\cbar }
  & = u_{\alpha\cbar ;\beta} + iu_{\alpha;0} \delta_{\beta\gamma} + R_{\beta\cbar \alpha\bar{\sigma}} u_{\sigma} \notag \\
  & = v_{\beta}  \delta_{\alpha\gamma} +v_{\alpha}  \delta_{\beta\gamma} +R_{\beta\cbar \alpha\bar{\sigma}} u_{\sigma}.
\end{align}
Taking the trace over $\beta$ and $\gamma$, we obtain
\begin{equation}\label{e:ricci1}
0 = (n+1)v_{\alpha}  + R_{\alpha\bar{\sigma}} u_{\sigma}.
\end{equation}

Let $f = |\bar{\partial}_{\flat} u|^2$, then
\begin{equation}
f_{\alpha} = (u_{\beta}u_{\bbar})_{;\alpha} = u_{\beta} u_{\bbar;\alpha} = \bar{v} u_{\alpha}.
\end{equation}
Therefore,
\begin{equation}\label{e:b41}
f_{\alpha\bbar} 
= (\bar{v}u_{\alpha})_{;\bbar}
= \bar{v}_{\bbar} u_{\alpha} +|v|^2 \delta_{\alpha\beta}.
\end{equation}
Taking conjugate and changing the role of $\alpha$ and $\beta$, we obtain,
\begin{equation}\label{e:b42}
f_{\bbar\alpha} = v_{\alpha}  u_{\bbar} + |v|^2 \delta_{\alpha\beta} .
\end{equation}
From \eqref{e:b41} and \eqref{e:b42}, we obtain
\begin{align}\label{e:a11}
if_0 \delta_{\alpha\beta} =
 f_{\alpha\bbar} - f_{\bbar \alpha}
  = \bar{v}_{\bbar} u_{\alpha}  - v_{\alpha}  u_{\bbar}.
\end{align}
We need the lemma below, whose proof is deferred to the end of this section.
\begin{lemma}\label{lem:la} Suppose that $U$ and $V$ are two vectors in $\mathbb{C}^n$, with $n\geq 2$. If
\begin{equation}\label{e:uotimesv}
U\otimes \bar{V} - V \otimes \bar{U} = \lambda I_{n\times n}
\end{equation}
where $I_{n\times n}$ is the identity matrix, then $\lambda =0$.
\end{lemma}

Applying Lemma~\ref{lem:la}, we find from \eqref{e:a11} that $f_0=0$ and $\bar{v}_{\bbar} u_{\alpha}  = v_{\alpha}  u_{\bbar}$. Hence
\begin{equation}\label{e:uv}
v_{\alpha}  = \eta_u u_{\alpha} 
\end{equation}
where the function $\eta_u$ is given by
\begin{equation}
\eta_u = |\partialb u|^{-2} (\bar{v}_{\bar{\beta}} u_{\beta}).
\end{equation}
Consequently, from \eqref{e:ricci1},
\begin{equation}\label{e:ricci}
R_{\alpha\bar{\sigma}} u_{\sigma} = -(n+1) v_{\alpha}  = -(n+1)\eta_u u_{\alpha} .
\end{equation}
This holds for all solution $\varphi$. Therefore, by the fully integrability assumption,
at each point $p'$ near $p$, the matrix $[R_{\alpha\bbar}|_{p'}]$ must have a single eigenvalue of multiplicity $n$. Therefore,
\begin{equation}\label{e:ricci2}
R_{\alpha\bbar} |_{p'}= -(n+1)\eta(p') \delta_{\alpha\beta}
\end{equation}
near $p$, where $\eta(p') = \eta_u(p')$ that does not depend on the choice of $u$. Consequently, $\theta$ is pseudo-Einstein near $p$.

We shall show that $\eta$ must be a constant. In fact, it is known that a pseudo-Einstein structure with vanishing pseudo-hermitian torsion must have constant scalar curvature (see, e.g. \cite{DT,W}). Indeed, since \(R_{\alpha\bbar} = \frac{R}{n}\delta_{\alpha\beta}\),
we obtain
\begin{equation}
R_{\alpha\bbar;\beta} = \tfrac{1}{n}R_{;\alpha}.
\end{equation}
On the other hand, from the Bianchi identity (see \cite[page 311]{DT} or \cite{L})
\begin{equation}
R_{;\alpha} = R_{\alpha\bbar;\beta} +i(n-1) A_{\alpha\beta,\bbar},
\end{equation}
we deduce, since $A_{\alpha\beta}=0$, that \(R_{;\alpha} = \frac{1}{n}R_{;\alpha} \),
and therefore, $R_{;\alpha} =0$ (as $n\geq 2$). This implies that $R$, being a real-valued CR function, must be a constant.
Therefore, from \eqref{e:ricci2}, we see that $\eta = -\frac{R}{n(n+1)}$ is a real constant that does \emph{not} depend on the choice of~$\varphi$.

On the other hand, from \eqref{e:wecur} and \eqref{e:uv}, we deduce that
\begin{equation}\label{e:443}
R_{\beta\cbar\alpha\bar{\sigma}} u_{\sigma}
=
\eta(u_{\beta}\delta_{\alpha\gamma} + u_{\alpha}\delta_{\beta\gamma}).
\end{equation}
For each fixed $\sigma$, by fully integrability again, we can choose a solution $\varphi$ such that $u_{\rho} = \delta_{\rho\sigma}$ at any given point near~$p$. Plugging this into~\eqref{e:443}, we obtain
\begin{equation}
R_{\beta\cbar\alpha\bar{\sigma}} 
= 
\eta(\delta_{\alpha\gamma} \delta_{\beta\sigma}+\delta_{\alpha\sigma}  \delta_{\beta\gamma}).
\end{equation}
This is precisely~\eqref{e:cstcur} which in particular implies that the Chern-Moser tensor $S_{\beta}{}^{\alpha}{}_{\gamma\bar{\sigma}}$ vanishes near~$p$, i.e., $M$ is CR spherical near~$p$ (by Cartan-Chern-Moser theory \cite{CM}, see also \cite[Section~1.5]{DT}).

We have three cases depending the sign of $\eta$.
\begin{enumerate}[(i)]
\item If $\eta =0$, then $(M,\theta)$  is locally equivalent to Heisenberg manifold $(\mathbb{H}^n, \Theta)$.
\item If $\eta >0$, then $(M,\theta)$ is locally homothetic to the model $Q_-(c)$ above.
\item If $\eta < 0$, then $(M,\theta)$ is locally homothetic to $(\mathbb{S}^n, \theta_0)$.
\end{enumerate}

The proof of sufficiency follows similar idea in \cite{OS}. Suppose that $(M,\theta)$ has constant curvature and vanishing pseudo-hermitian torsion. Then the Chern-Moser tensor of $M$ vanishes; $M$ is CR equivalent to an open set in Heisenberg manifold. Let $F:\mathbb{H}^n \to M$ be a local CR diffeomorphism, $F(0) = p\in M$. Put $\hat{\theta} = F^{\ast} \theta = e^{2\varphi} \Theta$. It suffices to show that the equation $B_{\hat{\theta}}(\sigma) =0$ is fully integrable.  Since $\hat{\theta}$ has constant curvature and vanishing pseudo-hermitian torsion, $B_{\Theta}(\varphi) = 0$. By Lemma~\ref{lem:additive}
\begin{equation}
B_{\Theta}(\varphi+\sigma)
=
B_{\Theta}(\varphi) + B_{\hat{\theta}}(\sigma)
=
B_{\hat{\theta}}(\sigma).
\end{equation}
Then the fully integrability of $B_{\hat{\theta}}$ follows immediately from that of $B_{\Theta}$, which was established in Proposition~\ref{prop:heisenberg}. The proof is complete.
\end{proof}
\begin{proof}[Proof of Lemma~\ref{lem:la}]
Observe that the matrix appearing on the left-hand side of the equation~\eqref{e:uotimesv} is of rank at most two, as it is the sum of two rank-1 matrices. Therefore, if $n\geq 3$ then both side of~\eqref{e:uotimesv} must be zero; in particular, $\lambda=0$, as desired. If $n=2$, we write 
$U=[a,b]^T$ and $V=[c,d]^T$ and rewrite the equation  as
\begin{equation}\label{e:m446}
\begin{bmatrix}
a\bar{c}-c\bar{a} & a \bar{d} - c \bar{b} \\
b \bar{c} - d \bar{a} & b \bar{d} - d \bar{b} 
\end{bmatrix}
=
\begin{bmatrix}
\lambda & 0 \\
0 & \lambda
\end{bmatrix}.
\end{equation}
If at least one of $a,b,c$ or $d$ is zero, then clearly $\lambda =0$. Otherwise, from equalities of off diagonal entries in \eqref{e:m446}, one has 
\(
a/c = \overline{b/d} : = k,
\)
for some $k\in \mathbb{C}$, so that $a = kc$ and $b = \bar{k} d$. From the equalities of diagonal entries
\begin{align}
\lambda & = a\bar{c} - c\bar{a} = |c|^2(k-\bar{k}),\\
\lambda  &= b\bar{d} - d\bar{b} = |d|^2(\bar{k} - k).
\end{align}
We deduce that \( (|c|^2 + |d|^2)(k-\bar{k}) =0 \). Therefore, $k-\bar{k} =0$ and  $\lambda =0$.
\end{proof}
\section{M\"{o}bius equation and stability group; Proof of Theorem~\ref{thm:compact1}}
In this section, we prove Theorems~\ref{thm:compact1}. We shall need several lemmas.
\begin{lemma}\label{thm:local1}
 Let $(M,\theta)$ be a pseudo-hermitian manifold of CR dimension $n$. Suppose that $\varphi$ is a nonconstant CR-pluriharmonic
 solution to $B_{\theta}(\varphi)=0$. Let $f=e^{-2\varphi}|\partialb \varphi |^2$ and $ X = H^{\theta}_f - fT$, where $H^{\theta}_f$ is the the \emph{contact Hamiltonian} with potential $f$ and $T$ is the Reeb field.
 Then $X$ is an infinitesimal CR automorphism which vanishes precisely at the critical points of $\varphi$. If in addition, $n\geq 2$, then $X$ is infinitesimal pseudo-hermitian transformation and $Xf=0$. 
\end{lemma}
Here by \emph{contact Hamiltonian} with potential $f$, we mean the unique vector field  $H^{\theta}_f$ satisfying
\begin{equation}
H^{\theta}_f \lrcorner\, \theta =0 \quad \text{and}\quad H^{\theta}_f \lrcorner\, d\theta = df - (Tf)\theta.
\end{equation}
In terms of local frame,
\begin{equation} 
H^{\theta}_{f} = if^{;\abar} Z_{\abar} - i f^{;\alpha} Z_{\alpha}  .
\end{equation} 
It is well-known that
$X: = H^{\theta}_f - fT$ is an infinitesimal contact transformation; in fact by Gray's theorem, any such transformation arises in this way. Moreover, $X$ is an infinitesimal CR automorphism if and only if the potential function $f$ satisfies
\begin{equation} 
f_{;\alpha\beta} + ifA_{\alpha\beta} =0,
\end{equation} 
where $A_{\alpha\beta}$ is the torsion of~$\theta$ (see \cite{L96}).
\begin{proof}[Proof of Lemma~\ref{thm:local1}]
As $\varphi$ is CR-pluriharmonic, we take $\psi$ such
that $G:=e^{-\varphi - i\psi}$. Let $u = e^{-\varphi}\cos(\psi) = \Re(G)$. Then, $u$ is CR-pluriharmonic and satisfies
\begin{equation} 
u_{\alpha\beta}=0, \quad u_{\alpha\bbar}  = v\delta_{\alpha\beta} , \quad P_{\alpha} u =0.
\end{equation} 
We deduce that
\begin{equation}
 0 = P_{\alpha}  u = u_{\cbar}{}^{\cbar } {}_{\alpha}  +inA_{\alpha\beta} u^{\beta} .
\end{equation}
Consequently,
\begin{equation}
 \bar{v}_{\alpha}  = \tfrac{1}{n} u_{\bar{\gamma}}{}^{\bar{\gamma}}{}_{\alpha} = -iA_{\alpha\beta}u^{\beta} .
\end{equation}
By direct calculation,
\begin{equation}
f = e^{-2\varphi} | \partialb \varphi |^2 = |\partialb u|^2,
\end{equation}
and thus
\begin{equation} \label{e:58}
 f_{\alpha}  = (u_{\beta}u_{\bbar   })_{;\alpha}  = u_{\beta}u_{\bbar;\alpha} =  \bar{v}u_{\alpha} .
\end{equation}
On the other hand, from $u_{\alpha\beta}  = 0$, we see that
\begin{equation} 
A_{\alpha\beta}u_{\gamma}  = A_{\alpha\gamma}u_{\beta} .
\end{equation} 
Therefore,
\begin{align}
 f_{;\alpha\beta} = (\bar{v}u_{\alpha} )_{;\beta}
  = \bar{v}_{\beta}  u_{\alpha} 
  = -i A_{\beta\gamma}u_{\cbar } u_{\alpha} 
  = -i A_{\alpha\beta}u_{\cbar } u_{\gamma} 
  = -i A_{\alpha\beta}f,
\end{align}
Hence,
\begin{equation}
 f_{;\alpha\beta} + i A_{\alpha\beta}f = 0.
\end{equation}
Therefore, $X$ is an infinitesimal CR automorphism.

Suppose that $n\geq 2$. We shall prove that $Xf=0$. The argument is similar to that in the proof of Theorem~\ref{thm:cstcur}. Observe that that from \eqref{e:58}
\begin{align}\label{e:b1}
f_{\alpha\bbar} 
=
(\bar{v} u_{\alpha})_{;\bbar}
=
\bar{v}_{\bbar} u_{\alpha} + \bar{v} u_{;\alpha\bbar}
=
\bar{v}_{\bbar} u_{\alpha}  + |v|^2 \delta_{\alpha\beta}.
\end{align}
Taking conjugate and changing the role of $\alpha$ and $\beta$, we obtain,
\begin{equation}\label{e:b2}
f_{\bbar\alpha} = v_{\alpha}  u_{\bbar} + |v|^2 \delta_{\alpha\beta} .
\end{equation}
From \eqref{e:b1} and \eqref{e:b2}, we obtain
\begin{align}\label{e:a1}
if_0 \delta_{\alpha\beta}  =
 f_{\alpha\bbar} - f_{\bbar \alpha}
 = \bar{v}_{\bbar} u_{\alpha}  - v_{\alpha}  u_{\bbar}.
\end{align}
Applying the Lemma~\ref{lem:la}, we have $f_0=0$ and thus
\begin{align}
X f 
= (H^f_{\theta} - fT) f
= fTf
= 0,
\end{align}
as desired. Next, from $f_0 =0$, we obtain
\begin{equation}
L_{X} \theta = X \lrcorner d\theta + d(X \lrcorner \theta) = -(Tf) \theta = 0,
\end{equation}
where $L_X$ is the Lie derivative. Therefore, $X$ generates a local flow of pseudo-hermitian isometries, as desired.

Finally, observe that $X_p = 0$ if and only if $f(p)=0$ if and only if $p$ is a critical point of $\varphi$. The proof is complete.
\end{proof}

\noindent
\textbf{Remark.} (i) In a recent paper \cite{IV2}, the authors consider the vector field defined (in our notations) by
\(
Q:= H_{\sigma}^{\theta} - \sigma T, \ \sigma = \frac{1}{2}d(e^{-2\varphi})(T).
\)
It is proved in Section 5.2 of \cite{IV2} that $Q$ is also an infinitesimal CR automorphism, \emph{provided} that the torsion vanishes (i.e., $A_{\alpha\beta}=0$). We point out here that $Q$ and $X$ have different geometric behaviors: on compact manifolds, $\langle \theta, Q\rangle = \sigma$ changes sign (unless $Q$ vanishes identically), but $\langle \theta, X\rangle = e^{-2\varphi}|\partial_{\flat} \varphi|^2 \geq 0$ does not. The last fact implies in particular that the angle between $X$ and $T$ (measured using Webster metric) is less than $\pi/2$ at most points. We thank an anonymous referee for drawing our attention to this interesting paper.

(ii) One can show, using this lemma, that near a regular point of $\varphi$ the Webster metric $g_{\theta}$ is a warped product metric, and hence, this lemma is similar to Brinkmann's in \cite{BR}. We leave the detail to the reader.
\begin{proof} [Proof of Theorem~\ref{thm:compact1}] Suppose, for contradiction, that $B_{\theta}(\varphi) =0$ had a nonconstant CR-pluriharmonic solution $\varphi$. Then $X$ (as defined in Lemma~\ref{thm:local1}) is an infinitesimal CR automorphism which vanishes precisely at the critical points of~$\varphi$. Since $M$ is compact and $\varphi$ is real-valued, the critical points of $\varphi$ must be nowhere dense and nonempty. If $p$ is such a critical point, then $\aut(M,p)$ has positive dimension. This is a contradiction.
\end{proof}
\section{Manifolds with nonpositive Webster's Ricci curvature and Proof of Theorem~\ref{thm:compact2}}
An essential ingredient for our proof of Theorem~\ref{thm:compact2} is the following Bochner-type identity discovered in \cite{CCY}. We remark that one can also use Greenleaf's Bochner formula (see \cite{IV,LW}) in the proof.
\begin{proposition}[Chanillo et al. \cite{CCY}]
 Let $f$ be a smooth function on $M$, then
 \begin{multline}\label{e:bochner}
  -\Box_\flat |\partialb f|^2 
  = \sum_{\alpha,\beta}\left(|f_{\abar\bbar}|^2 + |f_{\abar \beta}|^2 \right)
   - \tfrac{n+1}{n} \left(\Box_{\flat} f\right)_{\abar} \bar{f}_{\alpha}  - \tfrac{1}{n} f_{\abar} \left(\overline{\Box_{\flat} f}\right)_{\alpha} \\
   + R_{\alpha\bbar}  f_{\abar}f_{\beta} - \tfrac{1}{n} \bar{f}_{\alpha}  \overline{P_{\alpha} \bar{f}} + \tfrac{n-1}{n}f_{\abar} (P_{\alpha}  \bar{f}).
 \end{multline}
\end{proposition}
\noindent
Here, $\Box_{\flat}$ is the Kohn-Laplacian acting on functions by \(
 \Box_{\flat} f = - f_{\cbar }{}^{\cbar },
\)
and $P_{\alpha} $ is the third order operator defined in \eqref{e:palpha}.
\begin{proof}[Proof of Theorem~\ref{thm:compact2}]
Suppose that $\varphi$ satisfies $B_{\theta}(\varphi)=0$ on $M$. We need to show
that $\varphi$ is a constant. Locally, we can find a real-valued function $\psi$ such that $G: = e^{-\varphi - i\psi}$ is CR, i.e., $G_{\abar}=0$. Let $u = \Re (G) = e^{-\varphi}\cos(\psi)$ as before. We compute,
\begin{equation}\label{e:ua}
   u_{\alpha}  
 =\tfrac{1}{2}(G_{\alpha}  + \overline{G}_{\alpha} )
 = \tfrac{1}{2}G(-\varphi_{\alpha}  - i\psi_{\alpha} ) 
 = -G\varphi_{\alpha} .
\end{equation}
Therefore,  \( u_{\alpha\beta} = -\frac{1}{2}G B_{\theta}(\varphi)_{\alpha\beta}= 0\). We also compute, using $G_{\bbar}=0$,
\begin{equation}\label{e:63}
 u_{\alpha\bbar}  = (-G\varphi_{\alpha} )_{;\bbar} = -G\varphi_{\alpha\bbar}  = -\tfrac{1}{n}G(\overline{\Box_\flat \varphi}) \delta_{\alpha\beta} .
\end{equation}
Thus,
\begin{equation}
\sum_{\alpha,\beta} |u_{\alpha\bbar} |^2 
 = n \left(\tfrac{1}{n}\right)^2|G|^2|\Box_{\flat} \varphi |^2
 = \tfrac{1}{n} e^{-2\varphi}|\Box_{\flat} \varphi |^2.
\end{equation}
Also, taking the trace of \eqref{e:63},
\begin{equation}
 \Box_\flat u = \overline{G}\Box_\flat \varphi.
\end{equation}
Therefore,
\begin{equation}
(\Box_{\flat} u)_{;\bar{a}} 
 = \overline{G}_{\abar} \Box_{\flat} \varphi 
 + \overline{G}(\Box_{\flat}  \varphi )_{;\abar}.
\end{equation}
From this, we easily deduce that
\begin{align}
(\Box_{\flat} u)_{;\abar} u_{\alpha} 
 &= ( \overline{G}_{\abar} \Box_{\flat} \varphi 
   + \overline{G}(\Box_{\flat}  \varphi )_{;\abar})(-G\varphi_{\alpha} ) \notag \\
 &= -(G\overline{G}_{\abar}  \Box_{\flat} \varphi + |G|^2 ( \Box_{\flat} \varphi)_{;\abar})\varphi_{\alpha}  \notag \\
 & = -(|G|^2|  \Box_{\flat} \varphi )_{;\abar} \varphi_{\alpha}  \notag \\
 & = -(e^{-2\varphi}  \Box_{\flat} \varphi )_{;\abar} \varphi_{\alpha} .
\end{align}
On the other hand, from \eqref{e:ua}
\begin{equation}
 | \partialb u|^2 = |G|^2 \, |\partialb \varphi |^2 = e^{-2\varphi} |\partialb \varphi |^2.
\end{equation}
Plugging all into 
the Bochner-type identity~\eqref{e:bochner} (notice also that $u$ is CR-pluriharmonic and thus $P_{\alpha}  u =0$). 
\begin{align}
-\Box_{\flat}(e^{-2\varphi} |\bar{\partial}_{\flat} \varphi|^2)= -\Box_\flat | \partialb u|^2
 = 
 \tfrac{1}{n} e^{-2\varphi}&|\Box_{\flat} \varphi |^2 -
 \tfrac{n+1}{n}\left(e^{-2\varphi} \Box_\flat \varphi \right)_{;\abar}\varphi_{\alpha} \notag  \\
 & -
 \tfrac{1}{n}\varphi_{\abar}\left(e^{-2\varphi} \overline{\Box_\flat \varphi}\right)_{;\alpha}  + e^{-2\varphi} R_{\alpha\bbar} \varphi_{\alpha}  \varphi_{\bbar} .
\end{align}
Taking integral both sides on closed manifold $M$, we obtain
\begin{equation}
 0 = -\frac{n+1}{n}\int_M e^{-2\varphi} | \Box_\flat \varphi | ^2 + \int_M e^{-2\varphi} R_{\alpha\bbar} \varphi_{\alpha}  \varphi_{\bbar}.
\end{equation}
Since the Webster Ricci tensor is supposed to be nonpositive, it must hold that $\Box_\flat \varphi =0$. This implies that $\varphi$ is a real-valued CR function and therefore must be a constant.
\end{proof}

\begin{corollary}
Let $M$ be a compact strictly pseudo-convex CR manifold of dimension at least 5. Suppose that $\theta$ and $\hat{\theta}$ are two pseudo-hermitian structures on $M$ such that $\hat{A}_{\alpha\beta}  = A_{\alpha\beta}$ and $\hat{R}^{\circ}_{\alpha\bbar}  = R^{\circ}_{\alpha\bbar} $. Suppose further that one of the following holds,
\begin{enumerate}[(i)]
 \item For any $p\in M$, the stability group $\aut(M,p)$ is of zero dimension.
 \item Either  $\theta$ or $\hat{\theta}$ has nonpositive Webster Ricci curvature.
\end{enumerate}
Then $\hat{\theta}/\theta$ is a constant.
\end{corollary}
Here we remind  that $R^{\circ}_{\alpha\bbar}$ denotes the traceless Webster Ricci tensor. A recent related result \cite{Ho} shows that if $\theta$ and $\hat{\theta}$ are two pseudo-hermitian structures on $M$ such that satisfies $\hat{R}_{\alpha\bbar} = R_{\alpha\bbar}$ on a compact manifold, then $\hat{\theta}/\theta$ is a constant.

In Theorem~\ref{thm:compact2}, the nonpositivity of the Webster Ricci tensor is essential. In fact, we have seen that  in standard CR sphere, the solutions of $B_{\theta}(\varphi) =0$ contains nonconstant solutions \cite{JLCRYamabe,JLextremal}. On the other hand, when $(M,\theta)$ is assumed to be Einstein, then the CR sphere is the only one admitting a nonconstant solutions to $B_{\theta}(\varphi) =0$, as a recent result by Wang \cite{Wang} shows.  For the sake of completeness, we state
\begin{theorem}[\cite{JLextremal},\cite{Wang}]
 Let $(M,\theta)$ be a compact Einstein pseudo-hermitian manifold and $\hat{\theta} = e^{2\varphi}\theta$. Then the following hold.
\begin{enumerate}[(i)]
 \item The pseudo-hermitian structure $\hat{\theta}$ is M\"{o}bius  with respect to $\theta$ if and only if the Webster scalar curvature $\hat{R}$ of $\hat{\theta}$ is a constant.
 \item If $B_{\theta}(\varphi) = 0$ then $\varphi$ is a constant, unless $(M,\theta)$ is the CR sphere in $\mathbb{C}^{n+1}$. In the latter situation, $\varphi$ must be of the following form
\begin{equation}
\varphi(z) = -\log |\cosh t  + (\sinh t) z \cdot \bar{\mu}| + c
\end{equation}
where $t,c\in \mathbb{R}$, and $\mu \in \mathbb{S}^{2n+1}$. Moreover, $(\mathbb{S}^{2n+1},\hat{\theta})$ is locally equivalent to a pseudo-hermitian space form with vanishing torsion.
\end{enumerate}
\end{theorem}
\begin{proof} If $\hat{\theta} = e^{2\varphi}\theta$ is a M\"{o}bius structure with respect to $\theta$, then $\hat{\theta}$ is also Einstein, and hence of constant scalar curvature. The converse was proved by Jerison-Lee in \cite{JLextremal} (see also \cite{Wang}). Part (ii) follows
 from Wang's theorem in \cite{Wang}.
\end{proof}
\section{Examples}
Lemma~\ref{thm:local1} implies that the CR-pluriharmonic solutions to the M\"{o}bius equation  on a generic real hypersurface (e.g, a generic ellipsoid) are the constants, since there is no infinitesimal CR automorphism on such a hypersurface. A more interesting example is the following one, which is motivated by our recent work~\cite{ES15} on CR umbilical points.
\begin{example}\rm 
Consider the family of compact, homogeneous three dimensional CR manifolds given by
\begin{equation}
\mu_{\alpha} = \{[z_0 \colon z_1 \colon z_2] \in \mathbb{CP}^2 \mid |z_0|^2+|z_1|^2 + |z_2|^2 = \alpha |z_0^2 + z_1^2 +z_2^2|\}, \quad \alpha >1.
\end{equation}
Observe that $\mu_{\alpha}$ and its covers are the only nonspherical compact strictly pseudo-convex homogeneous manifolds in Cartan's classification \cite{Cartan}. In particular, every point on $\mu_{\alpha}$ is a nonumbilical point (in CR sense) and thus the local stability group $\aut(\mu_{\alpha},p)$ has zero dimension \cite{CM} (see also \cite{ES15}).
Therefore,
for any global pseudo-hermitian structure $\theta$ on $\mu_{\alpha}$ (or on any cover of $\mu_\alpha$), global CR-pluriharmonic solutions to
\(B_{\theta} (\varphi) = 0\)
are the constants. On the other hand,   since $\mu_{\alpha}$ is CR homogeneous, it is locally CR equivalent to a \emph{rigid} hypersurface in $\mathbb{C}^2$; there exists a locally defined pseudo-hermitian structure $\theta$ near every point such that $B_{\theta}(\varphi)$ has nonconstant CR-pluriharmonic solutions (see Example~\ref{ex:2} below.)
\end{example}
\begin{example}\label{ex:2}\rm 
Consider the rigid real hypersurface $M^{2n+1} \subset \mathbb{C}^{n+1}$ defined by 
\begin{equation}
F(Z,\bar{Z}) + \tfrac{i}{2}(w-\bar{w})=0,
\end{equation}
where $F$ has the following form
\begin{equation}
F(Z,\bar{Z}) = \Phi(z,\bar{z}) + |z_2|^2 + \dots + |z_n|^2, \quad \Phi(z,\bar{z}) = |z|^2 + \dots, Z=(z,z_2,\dots, z_n).
\end{equation}
We first choose the pseudo-hermitian structure
\begin{equation}
\eta
= \tfrac{1}{2}\left(ds +i \sum (F_{\abar} dz^{\abar} - F_{\alpha} dz^{\alpha})\right), \quad w = s+it,
\end{equation}
a  holomorphic frame 
\begin{equation}
Z_{\alpha} = \frac{\partial}{\partial z^{\alpha}} - 2i F_{\alpha} \frac{\partial}{\partial w},
\end{equation}
and its dual coframe $\theta^{\alpha} = \iota^{\ast} dz^{\alpha}$. Clearly,
\begin{equation}
d\theta = i F_{\alpha\bbar} dz^{\alpha} \wedge dz^{\bbar} = iF_{\alpha\bbar} \theta^{\alpha} \wedge \theta^{\bbar}.
\end{equation}
Therefore, the Levi metric is given by
$$
h_{1\bar{1}} = \Phi_{z\bar{z}}(z,\bar{z}),\quad  \text{and}\quad
h_{\alpha\bbar} 
= 
F_{\alpha\bbar}
=
\delta_{\alpha\beta} \quad \text{otherwise}.
$$
Observe that all Christoffel symbols vanish except $\Gamma_{11}^1$ and its conjugate:
\begin{equation}
\Gamma_{11}^{1} = \Phi_{z\bar{z}z}/\Phi_{z\bar{z}} = \partial_z \log(\Phi_{z\bar{z}}), \quad \Gamma^{\bar{1}}_{\bar{1}\bar{1}} = \overline{\Gamma^{1}_{11}}.
\end{equation}
Let $\theta = e^{2\sigma} \eta$, where
\[
\sigma = -\tfrac{1}{4}\log \Phi_{z\bar{z}}(z,\bar{z}).
\]
Suppose $\varphi(z)$ is a harmonic function near the origin in $\mathbb{C}_z$. Then $\varphi$ can be lifted to a CR-pluriharmonic function $\varphi^{\uparrow}$ on $M$ via parametrization $(Z,t) \mapsto (Z, t+i F(Z,\bar{Z}))$:
\[
\varphi(Z, t+ i F(Z,\bar{Z})) = \varphi(z).
\]
Direct calculation shows that
\[
B_{\theta}(\varphi^{\uparrow}) = 
2\left( \varphi_{zz} - 2(\varphi_z)^2\right)^{\uparrow} \theta^1 \otimes \theta^1\\
=
\frac{1}{4}S(f)^{\uparrow}  \theta^1 \otimes \theta^1.
\]
Here $S(f)$ is the Schwarzian derivative of a holomorphic mapping $f$ in $\mathbb{C}$ such that
\[
e^{2\varphi} = |f'|.
\]
Therefore,  $B_{\theta} (\varphi^{\uparrow}) =0$ if and only if $f$ is a (classical) M\"{o}bius transformation of complex plane, and hence,
\[
\varphi(z,\bar{z}) = -\log|az+b| + c.
\]
This example shows that the existence of a nonconstant solution to  $B_{\theta}(\varphi) = 0$ does not implies that $(M,\theta)$ has vanishing pseudo-hermitian torsion, neither does it implies $M$ is CR spherical (note, however, that $M\cap \{z=0\}$ is the Heisenberg hypersurface in $\mathbb{C}^{n}$.) 
\end{example}


\begin{thebibliography}{20}     

\bibitem{BER99}
Baouendi, M. Salah, Ebenfelt, P., and 
Rothschild, L.~P.: Real submanifolds in complex space and their mappings. Vol. 47. Princeton University Press, (1999).

\bibitem{Barletta}
Barletta, E.:
On the pseudohermitian sectional curvature of a strictly pseudoconvex CR manifold.
Diff. Geom. App. 
25 (2007) 612--631

\bibitem{BR}
Brinkmann, H.W.: 
Einstein spaces which are mapped conformally on each other.
Math. Ann.
94 (1925) 119--145.

\bibitem{Cartan}
Cartan, E.:
Sur la g\'{e}om\'{e}trie pseudo-conforme des hypersurfaces de l'espace de deux variables complexes.
Ann. Mat. Pura Appl., 
11(1) (1933):17--90

\bibitem{CY}
Case, J. and Yang, P.:
A Paneitz-type operator for CR-pluriharmonic functions.
Bull. Inst. Math.
Acad. Sinica (New Series)
Vol. 8 (2013), No. 3, pp. 285--322

\bibitem{CC}
Chang, S.-C., and Chiu, H.-L.:
Nonnegativity of CR Paneitz operator and its application to
the CR Obata's theorem.
J. Geom. Anal. 19 (2009), 261--287

\bibitem{CCY} Chanillo, S., Chiu, H.-L., Yang, P.: Embeddability for
3-dimensional Cauchy-Riemann manifolds and CR Yamabe invariants. Duke Math. J.
161 (2012), no. 15, 2909--2921.

\bibitem{CM} Chern, S.~S, Moser, J.: 
Real hypersurfaces in complex manifolds.
Acta Math.
133 (1974), 219--271

\bibitem{DT}
Dragomir, S. and Tommassini, G.: 
Differential Geometry and Analysis on CR manifolds, Progress in Mathematics, Volume 246, Birkh\"{a}user, (2006).

\bibitem{ES15}
Ebenfelt, P. and Duong, Son: 
Umbilical points on three dimensional strictly pseudoconvex CR manifolds. I. Manifolds with $U(1)$-action.
arXiv:1508.02612 [math.CV]

\bibitem{GL}Graham, C.~R. and Lee, J.~M.:
Smooth solutions of degenerate
Laplacians on strictly pseudoconvex domains. 
Duke Math. J. 
57 (1988), no. 3, 697--720.

\bibitem{Hi}
Hirachi, K.:
Scalar pseudo-Hermitian invariants and the Szeg\"{o} kernel on three-dimensional CR manifolds,
Complex geometry (Osaka, 1990), Lect. Notes in Pure and Appl. Math. 143 (1993), 67--76

\bibitem{IV}
Ivanov, S. and Vassilev, D.: 
 {An Obata type result for the first eigenvalue of the sub-Laplacian
on a CR manifold with a divergence free torsion}, J. Geom. Anal. 103 (2012), 475--504.

\bibitem{IV2}
\bysame,
The Lichnerowicz and Obata first eigenvalue theorems and the Obata uniqueness result in the Yamabe problem on CR and quaternionic contact manifolds. Nonlinear Anal. 126 (2015), 262--323

\bibitem{Ho}
Ho, P.-T: Rigidity in a conformal class of contact form on CR manifold. C. R. Math. Acad. Sci. Paris 353 (2015), p. 167--172.

\bibitem{JLCRYamabe}
Jerison, D. and Lee, J.: The Yamabe problem on CR manifolds.
J. Differential Geom.
Vol. 25 (2) (1987), 167--197.


\bibitem{JLextremal}
\bysame,
Extremals for the Sobolev inequality on the Heisenberg group and
the CR Yamabe problem. 
J. Amer. Math. Soc.
1 (1988), No. 1, 1--13


\bibitem{KR1} K\"{u}hnel, W., Rademacher, H. B.: 
Conformal diffeomorphisms preserving the Ricci tensor. 
Proceedings of the American Mathematical Society,
123(9) (1995): 2841--2848.

\bibitem{KR2} 
\bysame,
Liouville's theorem in conformal geometry,
J. math. pures et appl., 88.3 (2007): 251--260.

\bibitem{L86}
Lee, John M.: The Fefferman metric and pseudo-Hermitian invariants,
Trans. Amer. Math. Soc. 296 (1986), 411--429


\bibitem{L} 
\bysame,
Pseudo-Einstein structures on CR manifolds, 
Amer. J. Math.,
110.1 (1988), 157--178.

\bibitem{L96}
\bysame,
CR manifolds with noncompact connected automorphism groups.
J.  Geometric Analysis
(1996) 6(1), 79--90

\bibitem{FL}
Leitner, F.: On transversally symmetric pseudo-Einstein and Fefferman--Einstein spaces.
Math. Z.
(2007) 256, 443--459

\bibitem{LiLuk}
Li, S.-Y., and Luk, H.-S.:
An explicit formula for the Webster torsion of a pseudo-hermitian manifold and its application to torsion-free hypersurfaces.
Sci. China Series A: Math. 49.11, 1662--1682

\bibitem{LSW} Li, S.-Y.; Duong, Son, and Wang, X.:
A new characterization of the CR sphere 
and the sharp eigenvalue estimate for the Kohn Laplacian,
Advances in Mathematics, 281 (2015), pp. 1285--1305

\bibitem{LW} Li, S-Y. and Wang, X.: An Obata-type Theorem in CR Geometry. 
J. Differential Geom. 
95.3 (2013), 483--502

\bibitem{OS} Osgood, B. and Stowe, D.: The Schwarzian derivative and conformal mapping of Riemannian manifolds,
Duke Math. Journal
67 (1992), 57--99


\bibitem{Tanaka}
Tanaka, N.:
A differential geometric study on strongly pseudo-convex manifolds.
Lectures in Mathematics, Department of Mathematics,
Kyoto University, No. 9. Kinokuniya Book-Store
Co., Ltd., Tokyo, (1975) iv+158 pp.

\bibitem{Wang} Wang, X.:
On a remarkable formula of Jerison and Lee in CR geometry.
Mathematical Research Letters,
Volume 22 (2015),
Number 1,
Pages: 279--299

\bibitem{W} Webster, Sydney M.: Pseudo-Hermitian structures on a real hypersurface. 
J. Differential Geom. 13.1 (1978), 25--41.

\bibitem{Xu} Xu, X.: On the Existence and Uniqueness of Solutions of M\"{o}bius Equations.
Transactions of the American Mathematical Society
Vol. 337, No. 2 (Jun., 1993), pp. 927--945

\end{thebibliography}
\end{document}